\documentclass[11pt]{article}
\usepackage[lmargin=3cm,rmargin=3cm,tmargin=3cm,bmargin=30mm]{geometry}

\usepackage{amsthm,graphicx}
\usepackage{amsmath}

\newcommand{\todd}{totally odd $K_4$-subdivision}
\newcommand{\ac}{$\alpha$-critical}
\newcommand{\Ec}{E_{c}}
\newcommand{\asub}{critical subgraph}

\newtheorem{theorem}{Theorem}
\newtheorem{lemma}{Lemma}
\newtheorem{claim}{Claim}

\title{On a Theorem of Sewell and Trotter\thanks{This work was 
partially supported by the {\em Actions
de Recherche Concert\'ees (ARC)\,} fund of the {\em Communaut\'e
fran\c{c}aise de Belgique}.}
}

\author{
Samuel Fiorini\thanks{Universit\'e Libre de Bruxelles, D\'epartement de Math\'ematique,
c.p. 216,  B-1050 Bruxelles, Belgium,  sfiorini@ulb.ac.be.}
\and Gwena\"el Joret\thanks{Universit\'e Libre de Bruxelles, D\'epartement d'Informatique,
c.p. 212,  B-1050 Bruxelles, Belgium, gjoret@ulb.ac.be. G. Joret is a Research Fellow of the Fonds 
National de la Recherche Scientifique (F.R.S.--FNRS).}
}
\date{}

\begin{document}

\maketitle

\begin{abstract}
Sewell and Trotter proved that every connected {\ac} graph 
that is not isomorphic to $K_{1}, K_{2}$ or an odd cycle
contains a {\todd}. 
Their theorem implies an interesting min-max
relation for stable sets in graphs without {\todd s}.
In this note, we give a simpler proof of Sewell and Trotter's theorem.
\end{abstract}

\section{Introduction}

Graphs considered in this note are finite, simple, and undirected.
A graph $G$ is {\em \ac} if $\alpha(G-e) > \alpha(G)$ for every $e\in E(G)$, where
$\alpha(G)$ denotes the maximum cardinality of a stable set in $G$.
A subdivision of $K_{4}$ is {\em totally odd} if each edge of $K_{4}$
has been replaced with an odd-length path. 

Answering a question of Chv\'atal~\cite{C75}, Sewell and Trotter~\cite{ST93}
proved the following theorem.

\begin{theorem} [\cite{ST93}]
\label{th-ST}
Every connected {\ac} graph that is not isomorphic to $K_{1}, K_{2}$ or an odd cycle
contains a {\todd}.
\end{theorem}

As noted by Sewell and Trotter~\cite{ST95}, their result implies an interesting min-max
relation for the cardinality of a stable set in graphs having no {\todd} as a subgraph.
For an arbitrary graph $G$, denote by $\tilde \rho(G)$ the minimum cost of a family of 
vertices, edges and odd cycles covering $V(G)$, where 
the cost of a vertex or an edge is 1, the cost of an
odd cycle $C$ is $(|C| - 1)/2$, and the cost of a family is the sum of the costs of
its elements. Then clearly $\alpha(G) \le \tilde \rho(G)$.
Moreover, by Theorem~\ref{th-ST} we have $\alpha(G) = \tilde \rho(G)$ when $G$
has no {\todd}. (Indeed, it is always possible to find an {\ac} subgraph $G' \subseteq G$
with $\alpha(G') = \alpha(G)$ by removing some edges of $G$, and by 
Sewell and Trotter's theorem every component of $G'$ must be a vertex, an edge, or an odd cycle.)

A further consequence of Theorem~\ref{th-ST} is that we can efficiently find 
a maximum cardinality stable set
in a graph $G$ without {\todd}s. Roughly, $\alpha(G)$ equals then the optimum of a linear
program that can be solved in polynomial time, and by iteratively removing from $G$ any vertex
$v$ such that $\alpha(G-v)=\alpha(G)$ we eventually find a maximum stable set. We refer the 
interested reader to~\cite{S03B, ST95} for the details.

The main step of Sewell and Trotter's proof of Theorem~\ref{th-ST} 
consists in finding a {\todd} in the union of three carefully chosen 
odd cycles, by considering the various ways in which these odd cycles 
can intersect. A similar but more compact proof was given by 
Schrijver~\cite[pp. 1196---1199]{S03B}.

The purpose of this note is to present a new and simpler proof of
Sewell and Trotter's result. Our proof relies on the following two
ideas. First, we prove a strengthened version of Theorem~\ref{th-ST}. 
Second, we use the extra strength of the new statement to obtain a 
contradiction, essentially, by operating few local modifications on a
minimum counter-example. 

\begin{theorem} 
\label{th}
Let $G$ be a connected {\ac} graph that is not isomorphic to $K_{1}, K_{2}$, 
an odd cycle, nor to a {\todd}. Then 
\begin{itemize}
\item $G$ contains a {\todd}, and, moreover,
\item if $\{x_1,x_2,x_3\} \subseteq V(G)$ induces a triangle, then 
at least two of the three subgraphs $G-x_i$ contain a {\todd}.
\end{itemize}
\end{theorem}

\section{The Proof}

The following lemma summarizes some basic properties of {\ac} graphs 
which we will need; see for instance Lov\'asz~\cite{L93} or 
Lov\'asz and Plummer~\cite{LP86} for a proof.
\begin{lemma}
\label{lem-deg2}
Let $G$ be a connected {\ac} graph with $|V(G)| \ge 4$.
Then every vertex has degree at least 2.
Moreover, if $u\in V(G)$ has exactly two neighbors $v,w$ in $G$, then 
$v$ and $w$ are not linked;
the only common neighbor of $v$ and $w$ is $u$; and
contracting $uv$ and $uw$ results in another \ac{} graph.
\end{lemma}

Consider now an arbitrary graph $G$. 
The maximum degree of a vertex in $G$ is denoted by $\Delta(G)$.
We say that an edge $e$ of
$G$ is {\em critical} if $\alpha(G-e) > \alpha(G)$. Let $\Ec(G)$
denote the set of critical edges of $G$. We call a subgraph $G' 
\subseteq G$ a {\em \asub} of $G$ if $V(G')=V(G)$, $\alpha(G')=\alpha(G)$, 
the graph $G'$ contains every critical edge of $G$, and $G'$
is $\alpha$-critical (see Figure~\ref{fig-asub} for an example).
Any such subgraph can be obtained from $G$ by iteratively removing
some edge which is non critical in the current subgraph, as long 
such an edge exists. In particular, every graph $G$ has a {\asub}.

\begin{figure}
\centering
\includegraphics[width=3.5cm]{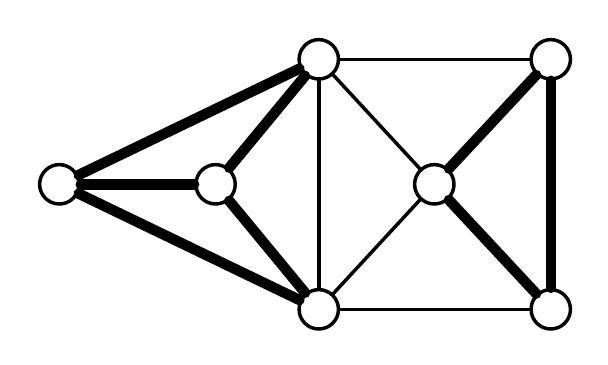}
\hspace{1cm}
\includegraphics[width=3.5cm]{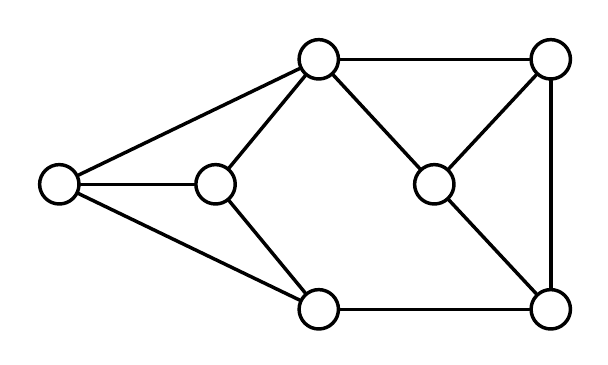}
\caption{\label{fig-asub} A graph $G$ with its critical edges drawn in bold (left), 
and a {\asub} of $G$ (right).}
\end{figure}

\begin{proof}[Proof of Theorem~\ref{th}]
Let $G$ be a counter-example with $|E(G)|$ minimum.
If there exists a minimum counter-example with a triangle,
we assume that $G$ has one. The outline of the proof is as
follows. After gathering some basic facts which will be 
repeatedly used subsequently, we split the proof into two
cases according to whether $G$ is triangle-free (Case I) or 
not (Case II). In both cases, we first construct a new graph $G'$
by locally modifying $G$. We then consider some critical subgraph
$H$ of $G'$ and choose a component $H'$ of $H$. Then $H'$ is a
connected \ac{} graph and not a counter-example. So we can apply 
the theorem to it. Finally we show that the theorem has to hold 
for $G$ too, which is a contradiction. In Case I, the new graph 
$G'$ is obtained by rotating some edge around one of its ends in 
such a way that a triangle appears. In Case II, the new graph $G'$ 
is obtained by adding an edge whose ends are both at distance $1$ 
from some triangle and then removing the three vertices of the 
triangle.

\begin{claim}
\label{claim-deg}
Every vertex $u$ of $G$ satisfies $\deg_{G}(u) \ge 3$.
\end{claim}
\begin{proof}
By Lemma~\ref{lem-deg2}, $G$ has no vertex with degree 1. Moreover, 
if $\deg_{G}(u)=2$ for some $u \in V(G)$, then using the same lemma
it follows that $u$ is not in a triangle, and that by contracting the 
two edges incident to $u$ we could obtain a smaller counter-example. 
\end{proof}

For $u \in V(G)$, let $G -_c u := (V(G-u), \Ec(G-u))$. Notice $\alpha(G-u)=\alpha(G)$, and 
hence 
\begin{equation}
\label{eq-Gu}
E(G -_c u)  =\{e\in \Ec(G): \exists S\subseteq V(G-u),  
S \textrm { is a maximum stable set in } G - e\}.
\end{equation}

\begin{claim}
\label{claim-Delta}
Let $u \in V(G)$. If $\Delta(G -_c u) \ge 3$, then $G-u$ contains a {\todd}.
\end{claim}
\begin{proof}
Let $H$ be a {\asub} of $G-u$. By definition, $G -_c u$ is a spanning
subgraph of $H$. If $\Delta(G -_c u) \ge 3$, then there is a component $H'$ of $H$
with $\Delta(H') \ge 3$. Clearly, every component of an {\ac} graph is also {\ac}, thus
$H'$ is {\ac}. Since $H'$ is not a counter-example to Theorem~\ref{th}, it must contain
a {\todd}, and so does $G-u$.
\end{proof}

\begin{claim}
\label{claim-uvw}
Let $u \in V(G)$. Any edge of $G$ not incident to $u$ such that one of its ends is adjacent to $u$
belongs to $E(G -_c u)$.
\end{claim}
\begin{proof}
Let $e$ denote the edge considered. Since $G$ is {\ac}, any maximum 
stable set of $G-e$ contains both of its ends and hence avoids $u$. 
Eq.~\eqref{eq-Gu} then implies that $e$ belongs to $E(G -_c u)$.
\end{proof}

{\bf \flushleft CASE I.}  
By Claim~\ref{claim-Delta}, we have $\Delta(G -_c u) \le 2$ for all $u\in V(G)$.
It follows then from Claim~\ref{claim-uvw} that 
\begin{itemize}
\item $G$ is cubic (3-regular), and 
\item $G$ has no subgraph isomorphic to $K_{2,3}$.
\end{itemize}
The graph $G$ must have two incident edges $uw,wv$ so 
that the only common neighbor of $u$ and $v$ is $w$.
Indeed, it is not difficult to check that the unique graph that is 
connected, cubic, triangle-free, not containing $K_{2,3}$ as a subgraph, and 
where every two incident edges lie 
in a common cycle of length 4 is the graph of the cube
(on 8 vertices), which is not {\ac}. Let $u_{1},u_{2}$ and
$v_{1},v_{2}$ be the two neighbors of $u$ and $v$, respectively, that are distinct from $w$. 
Let also $z$ be the neighbor of $w$ outside $\{u,v\}$.
 
By Claim~\ref{claim-uvw}, $uw \in E(G -_c v)$. Since $\Delta(G -_c v) \le 2$, we have
$uu_{1} \notin E(G -_c v)$ or $uu_{2} \notin E(G -_c v)$, 
say without loss of generality $uu_{2} \notin E(G -_c v)$.
Let $G' :=(G  - uu_{2}) + uv$. Using Eq.~\eqref{eq-Gu}, every maximum stable set of $G-uu_{2}$
contains $v$, hence $\alpha(G') = \alpha(G)$. This in turn implies
\begin{equation}
\label{eq-Gprime}
\{uv\} \cup E(G -_c u) \cup E(G -_c v) \subseteq \Ec(G').
\end{equation}

Let $H$ be an arbitrary {\asub} of $G'$ and denote by $H'$ the component 
containing $u$. Since $uw, wz \in E(G -_c v), vw \in E(G -_c u)$, it follows 
from Eq.~\eqref{eq-Gprime} that 
$\{u,v,w\}$ induces a triangle in $H'$ and $wz \in E(H')$. Lemma~\ref{lem-deg2}
then yields $uu_{1} \in E(H')$,  and  $vv_{1} \in E(H')$ or $vv_{2} \in E(H')$, say w.l.o.g.
$vv_{1} \in E(H')$. Using Claim~\ref{claim-uvw}, we have $e \in E(G -_c u)$ (resp.,
$e \in E(G -_c v)$) for every edge $e\neq uu_{1},vv_{1}$ which is incident to 
$u_{1}$ (resp., $v_{1}$) in $G$. Hence, by Eq.~\eqref{eq-Gprime}, $u,u_{1},v,v_{1},w$
have each degree at least 3 in $H'$, and in particular $H'$ is not isomorphic to 
a {\todd} (notice that $u_{1} \neq v_{1}$ by our choice of $u$ and $v$). 

As $|E(H')| \le |E(G)|$ and, by hypothesis, no minimum counter-example
has a triangle, we may apply the second part of Theorem~\ref{th} on $H'$ and 
triangle $\{u,v,w\}$, giving that  at least one of $H'-u, H'-v$ contains a {\todd}.  
Since that subdivision cannot use the edge $uv$, it also exists in $G$. 
This concludes the case where $G$ is triangle-free.

{\bf \flushleft CASE II.} Let $T=\{u,v,w\}$ be a triangle of $G$ such that 
both $G-u$ and $G-v$ contain no {\todd}.
By Claim~\ref{claim-Delta}, this implies $\Delta(G -_c u), \Delta(G -_c v) \le 2$, which 
in turn implies $\deg_{G}(x)=3$ for all $x\in T$, using Claim~\ref{claim-uvw}.
We will derive a contradiction by showing that 
$G-u$ or $G-v$ contains a {\todd}.

Suppose first that two distinct vertices $x,y\in T$ have a common neighbor
outside $T$. Then without loss of generality $x \in \{u,v\}$, and using $\deg_{G}(x)=3$, 
for every edge $e\in E(G)$ not incident to $x$ there exists a maximum stable set
in $G-e$ avoiding $x$. By Eq.~\eqref{eq-Gu}, this implies $G -_c x=G-x$.
Since $\Delta(G -_c x)=2$ and Theorem~\ref{th} applies to $G-x$, the latter 
graph is an odd cycle. 
Now, as $x$ is adjacent in $G$ to three consecutive vertices of this odd cycle, we
deduce that $G$ is a {\todd}, a contradiction. 
It follows that the neighbors outside $T$ of $u,v$ and $w$ are pairwise distinct; let us denote 
them respectively by $u',v'$ and $w'$.

Notice that any maximum stable set in $G - uu'$ must contain $v'$ and $w'$. 
In particular,  $\{u',v',w'\}$ is a stable set.
Let $G' := (G - T) + u'w'$. Using the previous remarks, it is easily seen 
that $\alpha(G') = \alpha(G) - 1$, which implies
\begin{equation}
\label{eq-other-Gprime}
\{u'w'\} \cup E((G -_c u')- T) \cup E((G -_c w') - T) \subseteq \Ec(G').
\end{equation}

Consider a {\asub} of $G'$, say $H$, and denote by $H'$ the component including $w'$. 
Since $\Delta(G -_c u) \le 2$ and $ww' \in E(G -_c u)$, the vertex $w'$ has a neighbor
$x$ outside $T$ such that $w'x \notin E(G -_c u)$. By Eq.~\eqref{eq-Gu}, this means 
that every maximum stable set in $G - w'x$ contains $u$. Hence, $w'x \in E((G -_c u') - T)$.
Also, using Claim~\ref{claim-uvw}, we have $e\in E((G -_c w') - T) $ for every edge 
$e \in E(G), e\neq w'x$ which is incident to $x$. 

Now, it follows from Eq.~\eqref{eq-other-Gprime} that $x$ has degree at least three in $H'$.
Since $|E(H')| < |E(G)|$, by applying Theorem~\ref{th} on $H'$
we deduce that the latter graph contains a {\todd} $K$.
We have $K \subset G-v$, unless $u'w' \in E(K)$. In the latter case, by
replacing the edge $u'w'$ of $K$ with the path $u'uww'$ we also obtain a {\todd} 
contained in $G-v$. This completes the proof of Theorem~\ref{th}. 
\end{proof}

To conclude, we mention that the second part of the statement of Theorem~\ref{th}
cannot be strengthened to ``the three subgraphs $G-x_i$ ($i = 1, 2, 3$) contain 
a {\todd}'', as illustrated by the rightmost graph in Figure~\ref{fig-asub}.

\bibliographystyle{plain}
\bibliography{sewell-trotter}

\end{document}